\RequirePackage{fix-cm}
\documentclass[smallextended]{svjour3}       % onecolumn (second format)
\smartqed  % flush right qed marks, e.g. at end of proof
\usepackage{graphicx}
\usepackage{amsmath}
\usepackage{amssymb}

\begin{document}

\title{Fractional Noether's Theorem with Classical and Caputo Derivatives: constants of motion for non-conservative systems%\thanks{Grants or other notes
}


\author{G.~S.~F.~Frederico         \and
        M.~J.~Lazo %etc.
}


\institute{G.~S.~F.~Frederico \at
              Department of Science and Technology, University of Cape Verde, Praia, Santiago, Cape Verde \\
              \email{gastao.frederico@ufsc.br}             \\
             \emph{Present address: Department of Mathematics, Federal University of Santa Catarina, Florian\'opilis, SC, Brazil}  %  if needed
           \and
           M.~J.~Lazo \at
              Institute of Mathematics, Statistics and Physics, Federal University of Rio Grande, Rio Grande, RS, Brazil.\\
              \email{matheuslazo@furg.br}
}

\date{Received: date / Accepted: date}

\maketitle

\begin{abstract}
Since the seminal work of Emmy Noether it is well know that all conservations laws in physics, \textrm{e.g.}, conservation of energy or conservation of momentum, are directly related to the invariance of the action under a family of transformations. However, the classical Noether's theorem can not yields informations about constants of motion for non-conservative systems since it is not possible to formulate physically meaningful Lagragians for this kind of systems in classical calculus of variation. On the other hand, in recent years the fractional calculus of variation within Lagrangians depending on fractional derivatives has emerged as an elegant alternative to study non-conservative systems. In the present work, we obtained a generalization of the Noether's theorem for Lagrangians depending on mixed classical and Caputo derivatives that can be used to obtain constants of motion for dissipative systems. In addition, we also obtained Noether's conditions for the fractional optimal control problem.
\end{abstract}

\section{Introduction}

The fractional calculus with derivatives and integrals of non-integer order started more
than three centuries ago, with l'H\^opital and Leibniz, when the derivative of order $1/2$
was suggested (see \cite{OldhamSpanier,CD:SaKiMa:1993,TenreiroMachado1,TenreiroMachado2} for the history of fractional calculus). 
This subject was then considered
by several mathematicians like Euler, Fourier, Liouville, Grunwald, Letnikov, Riemann, and
many others up to nowadays. Although the fractional calculus is almost as old as the usual
 integer order calculus, only in the last three decades it has gained more attention due to
 its many applications in various fields of science, engineering, economics,
 biomechanics, etc. (see \cite{HerrmannBook,CD:Hilfer:2000,Kilbas,Magin,SATM} for a review).
Actually, there are several definitions for fractional derivatives,
being the Riemann-Liouville and the Caputo the most popular definitions.
In special, the Caputo fractional derivative was introduced by Caputo and Mainardi 
in a seminal work \cite{CaputoMainardi} to model dissipation phenomenons.
Fractional derivatives are generally nonlocal operators and are historically
applied to study nonlocal or time dependent processes. In special, the first
and well established application of fractional calculus in Physics was in the
framework of anomalous diffusion observed in many physical
systems (e.g. in dispersive transport in amorphous semiconductor,
liquid crystals, polymers, proteins, etc \cite{Klages,GMMP,Metzler}).
Recently, the study of nonlocal quantum phenomena through fractional calculus
began a fast development, where the nonlocal effects are due to either long-range
interactions or time-dependent processes
with many scales \cite{CD:Hilfer:2000,Iomin,KBD,Laskin,Naber,Tarasov1,WBG}.
Relativistic quantum mechanics \cite{Herrmann,KP,MAB,Raspini,Zavada} and field
theories \cite{BI,Calcagni1,Lazo,Tarasov3,Vacaru} has been also
recently considered in the context of fractional calculus.

The fractional calculus of variation was introduced in the context of classical
mechanics. Riewe \cite{CD:Riewe:1997} showed that a Lagrangian involving fractional
time derivatives leads to an equation of motion with non-conservative forces such as
friction. It is a remarkable result since frictional and non-conservative forces are
beyond the usual macroscopic variational treatment \cite{Bauer}. Riewe
generalized the usual calculus of variations for a Lagrangian depending
on Riemann-Liouville fractional derivatives \cite{CD:Riewe:1997} in order
to deal with linear non-conservative forces. Actually, several approaches
have been developed to generalize the least action principle to include
problems depending on Caputo fractional derivatives, Riemann--Liouville
fractional derivatives, Riesz fractional derivatives and
others \cite{CD:Agrawal:2002,AT,BA,Cresson,LazoTorres1,OMT,MyID:227}
(see \cite{book:frac} for a recent review). Among theses approaches,
recently it was show that the action principle for dissipative systems
can be generalized, fixing the mathematical inconsistencies present in
the original Riewe's formulation, by using Lagrangians depending on classical
and Caputo derivatives \cite{LazoCesar}. The great importance of these results
is the fact that the calculus of variation with Lagrangians depending on both
classical and Caputo derivatives enable us to use all the mathematical machinery
of classical mechanics to study non-conservative systems.

Among the mathematical machinery of classical mechanics, the
Noether's theorem o f calculus of variation becomes one of the most
important theorems for physics in the 20th century. Since the
seminal work of Emmy Noether it is well know that all conservations
laws in mechanics, \textrm{e.g.}, conservation of energy or
conservation of momentum, are directly related to the invariance of
the action under a family of transformations. On the other hand,
non-conservative forces remove energy from the systems and, as a
consequence, the standard Noether constants of motion are broken. In
this context, the generalization of the Noether's theorem for the
fractional calculus of variation is fundamental to investigate the
action symmetries for non-conservative systems. Recently, it was
show that it is still possible to obtain Noether-type theorems for
fractional calculus of variations which cover both conservative and
nonconservative cases
\cite{El-Nabulsi,MR2351636,MR2405377,CD:FredericoTorres:2007,Caputo,CD:FredericoTorres:2006,FredericoTorres:2010,Zhang}.
In the present work, we generalize Noether's theorem for
Lagrangians depending on mixed classical and Caputo derivatives. 
It is important to stress that our results are based in the classical 
notion of conserved quantity $C$, that is, the classical derivative 
of such a quantity is equal to zero ($dC/dt=0$). It is a different 
approach from previous works 
\cite{MR2405377,CD:FredericoTorres:2007,Caputo,CD:FredericoTorres:2006,FredericoTorres:2010}
where it was introduced the notion of fractional-conserved quantity, 
where the classical derivative is substituted by a bilinear fractional 
operator $D$ ($D(C)=0$). Consequently, our present work is free from 
the difficulties introduced by  the notion of fractional-conserved 
quantity (see \cite{Bourdin} for a detailed discussion). Furthermore,
the generalized Noether's theorem we obtain enable us to investigate
constants of motion  for dissipative systems in the context of the
action principle formulated in \cite{LazoCesar}. As an example of
application to non-conservative systems, we study the problem of a
particle under a frictional force. Furthermore, we also generalize
the Noether-type theorems for the optimal control problem with
classical and Caputo derivatives.

The paper is organized in the following way. In Section \ref{sec:fdRL} we review the basic notions of Riemann-Liouville and Caputo fractional calculus, that are needed for formulating the fractional problem of the calculus of variations. The Euler-Lagrange equation and the Noether's theorem for Lagrangians depending on mixed classical and Caputo derivatives are obtained in Section \ref{sec:delay}. An example of application of the Noether's theorem for a particle under a frictional force is presented in Section \ref{example}. In Section \ref{foc} we generalize the Noether's theorems for the optimal control problem with classical and Caputo derivatives. Finally, the conclusions are presented in Section \ref{conclusion}.
%Finally, Section $5$ presents the conclusions.

%%%%%%%%%%%%%%%%%%%%%%%%%%%%%%%%%%%%%%%%%

\section{Preliminaries on Fractional Calculus}
\label{sec:fdRL}

In this section we fix notations by collecting the definitions and
properties of fractional integrals and derivatives needed in the
sequel
\cite{CD:Agrawal:2007,CD:MilRos:1993,CD:Podlubny:1999,CD:SaKiMa:1993}.

\begin{definition}(Riemann--Liouville fractional integrals)
Let $f$ be a continuous function in the interval $[a,b]$. For $t \in
[a,b]$, the left Riemann--Liouville fractional integral
$_aI_t^\alpha f(t)$ and the right Riemann--Liouville fractional
integral $_tI_b^\alpha f(t)$ of order $\alpha$,  are defined by
\begin{gather*}
_aI_t^\alpha f(t) = \frac{1}{\Gamma(\alpha)}\int_a^t
(t-\theta)^{\alpha-1}f(\theta)d\theta,\\
_tI_b^\alpha f(t) = \frac{1}{\Gamma(\alpha)}\int_t^b
(\theta-t)^{\alpha-1}f(\theta)d\theta,
\end{gather*}
where $\Gamma$ is the Euler gamma function and $0 <\alpha < 1$.
\end{definition}

\begin{definition}(Fractional derivatives in the sense of
Riemann--Liouville) Let $f$ be a continuous function in the interval
$[a,b]$. For $t \in [a,b]$, the left Riemann--Liouville fractional
derivative $_aD_t^\alpha f(t)$ and the right Riemann--Liouville
fractional derivative $_tD_b^\alpha f(t)$ of order $\alpha$ are
defined by
\begin{equation*}
\begin{split}
_aD_t^\alpha f(t)&={\frac{d}{dt}}\, {_aI_t^{1-\alpha}} f(t)\\
&= \frac{1}{\Gamma(1-\alpha)}\frac{d}{dt} \int_a^t
(t-\theta)^{-\alpha}f(\theta)d\theta
\end{split}
\end{equation*}
and
\begin{equation*}
\begin{split}
_tD_b^\alpha f(t) &= {-\frac{d}{dt}}\,{_tI_b^{1-\alpha}}f(t)\\
&= \frac{-1}{\Gamma(1-\alpha)}\frac{d}{dt} \int_t^b(\theta
-t)^{-\alpha}f(\theta)d\theta.
\end{split}
\end{equation*}
\end{definition}

\begin{definition}(Fractional derivatives in the sense of
Caputo) Let $f$ be a continuously  differentiable function in the
interval $[a,b]$. For $t \in [a,b]$, the left Caputo fractional
derivative ${_a^CD_t^\alpha} f(t)$ and the right Caputo fractional
derivative ${_t^CD_b^\alpha} f(t)$ of order $\alpha$ are defined by

\begin{equation*}
\begin{split}
_a^CD_t^\alpha f(t)&=\, {_aI_t^{1-\alpha}}{\frac{d}{dt}} f(t)\\
&= \frac{1}{\Gamma(1-\alpha)} \int_a^t
(t-\theta)^{-\alpha}{\frac{d}{d\theta}}f(\theta)d\theta
\end{split}
\end{equation*}
and
\begin{equation*}
\begin{split}
_t^CD_b^\alpha f(t) &= \,{_tI_b^{1-\alpha}}\left({-\frac{d}{dt}}\right)f(t)\\
&= \frac{-1}{\Gamma(1-\alpha)} \int_t^b(\theta
-t)^{-\alpha}\frac{d}{d\theta}f(\theta)d\theta.
\end{split}
\end{equation*}
\end{definition}

\begin{remark} If the Riemann--Liouville and the Caputo fractional
derivatives exist, then they are connected by the following
relations:
\begin{equation*}
 _a^CD_t^\alpha f(t) := {_aD_t^\alpha} \big( f- f(a) \big) (t)
\quad \Big( \text{resp.} \quad  \; _t^CD_b^\alpha f(t) :=
{_tD_b^\alpha} \big( f - f(b) \big) (t) \Big)\,.
\end{equation*}

Let us note that if $f(a)=0$ (resp. $f(b)=0$), then $_a^CD_t^\alpha
f(t) = {_aD_t^\alpha} f(t)$ (resp. ${_t^CD_b^\alpha} f(t) =
{_tD_b^\alpha}f(t)$).
\end{remark}

\begin{remark}
The Caputo fractional derivative of a constant is always equal to
zero. This is not the case of the strict fractional derivative in
the Riemann--Liouville sense.
\end{remark}

\begin{remark}
In the \textit{classical case} $\alpha = 1$, the fractional
derivatives of Riemann--Liouville
 and Caputo both coincide with the classical derivative. Precisely, modulo a $(-1)$ term in
 the right case, we have ${_aD_t^1} = {_a^CD_t^1}
  = - {_t^CD_b^1} = - {_tD_b^\alpha} = d/dt$.
\end{remark}

\begin{theorem}
Let $f$ and $g$ be two continuously differentiable functions
on $[a,b]$. Then, for all $t \in [a,b]$, the following property holds:
$$
{_a^CD_t^\alpha}\left(f(t)+g(t)\right)
= {_a^CD_t^\alpha}f(t)+{_a^CD_t^\alpha}g(t).
$$
\end{theorem}

We now present the integration by parts formula for fractional
derivatives.

\begin{lemma}
If $f$, $g$, and the fractional derivatives ${_a^C D_t^\alpha} g$
and $_tD_b^\alpha f$ are continuous at every point $t\in[a,b]$, then
\begin{equation*}
\int_{a}^{b}g(t)\cdot {_a^C D_t^\alpha}f(t)dt =\int_a^b f(t)\cdot
{_t D_b^\alpha} g(t)dt+\left[{_tI_b^{1-\alpha}}g(t) \cdot
f(t)\right]_{t=a}^{t=b}
\end{equation*}

for any $0<\alpha<1$. Moreover, if $f$ is a function such that
$f(a)=f(b)=0\,,$ we have simpler formula:

\begin{equation}
\label{integracao:partes}\int_{a}^{b}g(t)\cdot {_a^C
D_t^\alpha}f(t)dt =\int_a^b f(t)\cdot {_t D_b^\alpha} g(t)dt\,.
\end{equation}
\end{lemma}

\begin{remark}
We note that
formula \eqref{integracao:partes} is still
valid for $\alpha=1$ provided $f$ or $g$
are zero at $t=a$ and $t=b\,.$
\end{remark}

% ------------------------------------------------

\section{Main results: Euler--Lagrange equations and Noether's theorems for variational problems with classical and Caputo derivatives}
 \label{sec:delay}

In Section~\ref{sub1} we prove two important results for variational
problems: a necessary optimality condition of Euler--Lagrange type
(Theorem~\ref{Thm:FractELeq1}) and a Noether-type theorem
(Theorem~\ref{theo:tndf}). The results are then extended in
Section~\ref{foc} to the more general setting of optimal control.

% ------------------------------------------------

\subsection{Fractional variational problems with classical and Caputo derivatives}
\label{sub1}

We begin by formulating the fundamental problem in Lagrange form
under investigation.

\begin{problem}
\label{Pb1} The fractional problem of the calculus of variations
with classical and Caputo derivatives in Lagrange form consists to
find the stationary functions of the functional
\begin{equation}
\label{Pf} I[q(\cdot)] = \int_a^b
L\left(t,q(t),\dot{q}(t),{_a^CD_t^\alpha} q(t)\right) dt \,
\tag{$P_{C}$}
\end{equation}
subject to given appropriate boundary conditions, where $[a,b]
\subset \mathbb{R}$, $a<b$, $0 < \alpha< 1$,
$\dot{q}=\frac{dq}{dt}$, and the admissible functions $q: t \mapsto
q(t)$ and the Lagrangian $L: (t,q,v,v_l) \mapsto L(t,q,v,v_l)$ are
assumed to be $C^2$:
\begin{gather*}
q(\cdot) \in C^2\left([a,b];\,\mathbb{R}^n \right)\text{;}\\
L(\cdot,\cdot,\cdot,\cdot) \in
C^2\left([a,b]\times\mathbb{R}^n\times
\mathbb{R}^n\times\mathbb{R}^n ;\,\mathbb{R}\right)\text{.}
\end{gather*}
\end{problem}

Functional of the kind \eqref{Pf} with mixed integer order and Caputo fractional derivatives was previously considered in \cite{LazoCesar,OMT}. However, a Noether-type theorem for \eqref{Pf} is not yet considered in the literature. Furthermore, despite that the fundamental problem \eqref{Pf} could easily be generalized for $\alpha >1$, we choose $0<\alpha \leq 1$ for simplicity. Along the work, we denote by $\partial_{i}L$ the partial derivative
of $L$ with respect to its $i$th argument, $i=1,\ldots,4\,.$

% ------------------------------------------------

\subsubsection{Fractional Euler--Lagrange equations}

The Euler--Lagrange necessary optimality condition
 is central in achieving the main results of this work.
 Our results are formulated and proved using the Euler--Lagrange equations \eqref{eq:eldf}.

\begin{definition} (Space of variations) We denote by $S_{h}(a,b)$
the set of functions $h(\cdot)\in C^2\left([a,b];\,\mathbb{R}^n
\right)$ such that $h(a)=h(b)=0$\,.
\end{definition}

\begin{definition}
The funcional $I[(\cdot)]$ is $S_{h}(a,b)$-differentiable on a curve
$q(\cdot)\in C^2\left([a,b];\,\mathbb{R}^n \right)$ if and only if
its \emph{Fréchet differential}
$$\lim_{\epsilon\rightarrow 0}\frac{I[q+\epsilon h]-I[q]}{\epsilon}$$
exists in any direction $h(\cdot)\in S_{h}(a,b)$, then $DI$ is
called its differential and is given by
$$DI[q](h)=\lim_{\epsilon\rightarrow 0}\frac{I[q+\epsilon h]-I[q]}{\epsilon}\,.$$
\end{definition}

\begin{definition}(Fractional $S_h$-extremal with classical and Caputo derivatives).
\label{df1} We say that $q(\cdot)$ is an $S_h$-extremal with
classical and Caputo derivatives for funcional \eqref{Pf} if for any
$h(\cdot)\in S_{h}(a,b)$
$$DI[q](h)=0\,.$$
\end{definition}

\begin{theorem}
The differential of $I[(\cdot)]$ on $q(\cdot)\in
C^2\left([a,b];\,\mathbb{R}^n \right)$ is given by
\begin{multline}
\label{pel1} DI[q](h)=\int_a^b\Bigl[\partial_{2}
L\left(t,q,\dot{q},{_a^CD_t^\alpha} q\right)\cdot h+\partial_{3}
L\left(t,q,\dot{q},{_a^CD_t^\alpha}
q \right)\cdot \dot{h}\\
+\partial_{4} L\left(t,q,\dot{q},{_a^CD_t^\alpha} q
\right)\cdot{_a^CD_t^\alpha}h\Bigr] dt\,.
\end{multline}
\end{theorem}

\begin{proof}
We obtain equation \ref{pel1} by direct computations with help of a
Taylor expansion.
\end{proof}

We now obtain the fractional Euler--Lagrange necessary optimality condition.

\begin{theorem}(Fractional least-action principle).
\label{Thm:FractELeq1}
If $q(\cdot)$ is a  $S_h$-extremal to
Problem~\ref{Pb1}, then it satisfies the following
\emph{Euler--Lagrange equation with classical and Caputo derivatives}:
\begin{multline}
\label{eq:eldf}
\partial_{2} L\left(t,q(t),\dot{q}(t),{_a^CD_t^\alpha} q(t)\right)
-\frac{d}{dt}\partial_{3} L\left(t,q(t),\dot{q}(t),{_a^CD_t^\alpha}
q(t)\right) \\+ {_tD_b^\alpha}\partial_{4}
L\left(t,q(t),\dot{q}(t),{_a^CD_t^\alpha} q(t)\right)  = 0,\quad t \in
[a,b]\,.
\end{multline}
\end{theorem}

\begin{remark}
If $\alpha=1$, Problem~\ref{Pb1} is reduced to the classical problem
of the calculus of variations,
\begin{equation}
\label{eq:pbcv}
 I[q(\cdot)] = \int_a^b
F\left(t,q(t),\dot{q}(t)\right) \longrightarrow \min
\end{equation}
with $F\left(t,q(t),\dot{q}(t)\right):=L\left(t,q(t),\dot{q}(t),\dot{q}(t)\right)$, and one obtains from Theorem~\ref{Thm:FractELeq1} the standard Euler--Lagrange equations \cite{Logan:b}:
\begin{equation}
\label{eq:EL}
\partial_2 F\left(t,q,\dot{q}\right)=\frac{d}{dt}\partial_3 F\left(t,q,\dot{q}\right)\,.
\end{equation}
\end{remark}

\begin{remark}
Our variational Problem~\ref{Pb1} only involves Caputo
fractional derivatives but both Caputo and Riemann-Liouville
fractional derivatives appear in the necessary optimality condition
given by Theorem~\ref{Thm:FractELeq1}. This is different from
\cite{CD:Agrawal:2002,CD:FredericoTorres:2007} where
the necessary conditions only involve the same type of derivatives
(Riemann-Liouville) as those in the definition of the fractional
variational problem.
\end{remark}

\begin{proof} (of Theorem~\ref{Thm:FractELeq1})
According with Definition~\ref{df1},
a necessary condition for $q$ to be a $S_h$-extremal is given by
\begin{multline}
\label{pel}
\int_a^b\Bigl[\partial_{2} L\left(t,q,\dot{q},{_a^CD_t^\alpha}
q\right)\cdot h+\partial_{3} L\left(t,q,\dot{q},{_a^CD_t^\alpha}
q \right)\cdot \dot{h}\\
+\partial_{4} L\left(t,q,\dot{q},{_a^CD_t^\alpha}
q \right)\cdot{_a^CD_t^\alpha}h\Bigr] dt=0\,.
\end{multline}
Using the fact that $h\in S_{h}(a,b)$, and the classical
and Caputo \eqref{integracao:partes}
integration by parts formulas in the second and third
terms of the integrand of \eqref{pel}, respectively, we obtain
\begin{multline*}
\int_a^b\Bigl[\partial_{2} L\left(t,q,\dot{q},{_a^CD_t^\alpha}
q\right)-\frac{d}{dt}\partial_{3} L\left(t,q,\dot{q},{_a^CD_t^\alpha} q\right)\\
+ {_tD_b^\alpha}\partial_{4} L\left(t,q,\dot{q},{_a^CD_t^\alpha}
q\right)\Bigr]\cdot h \,dt=0.
\end{multline*}
Equality \eqref{eq:eldf} follows from the application
of the fundamental lemma of the calculus of variations
(see, \textrm{e.g.}, \cite{CD:Gel:1963}).
\end{proof}

% ------------------------------------------------

\subsubsection{Fractional Noether's theorem}

A classical result of Emmy Noether provides a relation between
groups of symmetries of a given equation and constants of motion,
i.e. first integrals. Precisely, if a Lagrangian system is invariant
under a group of symmetries then it admits an explicit conservation
law.

 The symmetries are defined via the action of one parameter group of
diffeomorphisms as follows

\begin{definition}(Group of symmetries)
\label{sy} For any real $\varepsilon\,,$ let
$\psi(\varepsilon,\cdot)\,:\mathbb{R}^n\rightarrow\mathbb{R}^n$ be a
diffeomorphism. Then
$\Psi=\left\{\psi(\varepsilon,\cdot)\right\}_{\varepsilon\in\mathbb{R}}$
is a one parameter group of diffeomorphisms of $\mathbb{R}^n$ if it
satisfies:
\begin{enumerate}
\item $\psi (0,\cdot) = Id_{\mathbb{R}^n}$;
\item $\forall \varepsilon,\varepsilon' \in \mathbb{R},
\; \psi (\varepsilon,\cdot) \circ \psi (\varepsilon',\cdot) = \psi
(\varepsilon+\varepsilon',\cdot) $;
\item $\psi(\cdot,\cdot)$ is of class $C^2$ with respect to $\varepsilon\,.$
\end{enumerate}
\end{definition}

Usual examples of one parameter groups of diffeomorphisms are given by translations in a given directions $v$
$$\psi:q \longmapsto q+\varepsilon v\,,\quad q\in\mathbb{R}^n$$
and rotations of angle $\omega$
$$\psi:q \longmapsto qe^{i\varepsilon\omega}\,,\quad q\in\mathbb{C}\,.$$

In \cite{MR2351636,MR2405377} the authors use the related notion of
one parameter family of infinitesimal transformations, instead of
group of diffeomorphisms. They are obtained using a Taylor expansion
of $y_t(\varepsilon)=\psi(\varepsilon,q(t))$ in a neighborhood of
$0\,.$ We obtain
$$y_t(\varepsilon)=\psi(0,q(t))+\varepsilon\frac{\partial \psi}{\partial \varepsilon}(0,q(t))+o(\varepsilon)\,.$$
Having in mind that $\psi (0,\cdot) = Id_{\mathbb{R}^n}\,,$ we
deduce that an infinitesimal transformation is of the form
$$q(t)\longmapsto q(t) + \varepsilon\xi(t,q(t)) + o(\varepsilon)$$
where $\frac{\partial \psi}{\partial
\varepsilon}(0,q(t))=\xi(t,q(t))\,.$

In order to prove a fractional Noether's theorem for
Problem~\ref{Pb1} we adopt a technique used in
\cite{CD:FredericoTorres:2007,CD:Jost:1998}. The proof is done in
two steps: we begin by proving a Noether's theorem without
transformation of the time (without transformation of the
independent variable); then, using a technique of
time-reparametrization, we obtain Noether's theorem in its general
form.

The action of one parameter group of diffeomorphisms on a Lagrangian
allows to define the notion
 of a symmetry for a fractional functional \eqref{Pf}

\begin{definition}(Invariance without transforming the time).
\label{def:inv1:MR} Functional \eqref{Pf} is said to be
$\varepsilon$-invariant under the action of one parameter group of
diffeomorphisms
$\Psi_2=\left\{\psi_2(\varepsilon,\cdot)\right\}_{\varepsilon\in\mathbb{R}}$
of $\mathbb{R}^n$ if it satisfies for any solution $q(\cdot)$ of
\eqref{eq:eldf}
\begin{multline}
\label{eq:invdf}
\int_{t_{a}}^{t_{b}} L\left(t,q(t),\dot{q}(t),{_a^CD_t^\alpha q(t)}\right) dt\\
= \int_{t_{a}}^{t_{b}}
L\left(t,\psi_2(\varepsilon,q(t)),\frac{d\psi_2}{dt}(\varepsilon,q(t)),{_a^CD_t^\alpha
\psi_2(\varepsilon,q(t))}\right) dt
\end{multline}
for any subinterval $[{t_{a}},{t_{b}}] \subseteq [a,b]$.
\end{definition}

The next theorem establishes a necessary condition of invariance.

\begin{theorem}(Necessary condition of invariance).
If functional \eqref{Pf} is invariant, in the sense of
Definition~\ref{def:inv1:MR}, then
\begin{multline}
\label{eq:cnsidf11} \frac{\partial \psi_2}{\partial
\varepsilon}(0,q(t))\cdot
\frac{d}{dt}\partial_{3} L\left(t,q(t),\dot{q}(t),{_a^CD_t^\alpha q(t)}\right)\\
+ \partial_{3}L\left(t,q(t),\dot{q}(t),{_a^CD_t^\alpha q(t)}\right)\cdot
\frac{d}{dt}\frac{\partial \psi_2}{\partial \varepsilon}(0,q(t))\\
+ \partial_{4} L\left(t,q(t),\dot{q}(t),{_a^CD_t^\alpha q(t)}\right)
\cdot {_a^CD_t^\alpha \frac{\partial \psi_2}{\partial \varepsilon}(0,q(t))}\\
- \frac{\partial \psi_2}{\partial \varepsilon}(0,q(t))\cdot
{_tD_b^\alpha}\partial_{4} L\left(t,q(t),\dot{q}(t),{_a^CD_t^\alpha
q(t)}\right) = 0 \, .
\end{multline}
\end{theorem}

\begin{proof}
As condition \eqref{eq:invdf} is valid for any subinterval
$[{t_{a}},{t_{b}}] \subseteq [a,b]$, we can get rid of the integral
signs in \eqref{eq:invdf}. Differentiating this condition with respect to $\varepsilon$, substituting $\varepsilon=0$,
the usual chain rule for the classical derivatives implies

\begin{multline}
\label{eq:SP} 0 = \partial_{2} L\left(t,q,\dot{q},{_a^CD_t^\alpha}
q\right)\cdot\frac{\partial \psi_2}{\partial \varepsilon}(0,q)
+\partial_{3}L\left(t,q,\dot{q},{_a^CD_t^\alpha}
q\right)\cdot\frac{\partial}{\partial\varepsilon}\left[ \frac{d \psi_2}{dt}(\varepsilon,q)\right]\vert_{\varepsilon=0} \\
+ \partial_{4} L\left(t,q,\dot{q},{_a^CD_t^\alpha} q\right)
\cdot\frac{\partial}{\partial\varepsilon}\left[ {_a^CD_t^\alpha}
\psi_1(\varepsilon,q)\right]\vert_{\varepsilon=0} \,.
\end{multline}
Using the definitions and properties of the Caputo fractional
derivatives given in Section~\ref{sec:fdRL} and the fact $d/dt$ and
$_a^CD_t^\alpha$  act on variable $t$ and
$\partial/\partial\varepsilon$ on variable $\varepsilon$, and
$\psi_2(\varepsilon,q)\in C^2$ with respect to $\varepsilon$ (see
Definition~\ref{sy}), we deduce that
 \begin{equation}
 \label{ce}
 \frac{\partial}{\partial\varepsilon}\left[ \frac{d \psi_2}{dt}(\varepsilon,q)\right]\mid_{\varepsilon=0}
 =\frac{d}{dt}\frac{\partial\psi_2}{\partial\varepsilon}(0,q)
 \end{equation}
 and
 \begin{equation}
 \label{ce1}
 \frac{\partial}{\partial\varepsilon}\left[ {_a^CD_t^\alpha} \psi_2(\varepsilon,q)\right]\mid_{\varepsilon=0}
 ={_a^CD_t^\alpha}\frac{\partial\psi_2}{\partial\varepsilon}(0,q)\,.
 \end{equation}
Substituting the quantities  \eqref{ce} and \eqref{ce1} into \eqref{eq:SP}, and using the Euler--Lagrange equation \eqref{eq:eldf}, the necessary
condition of invariance \eqref{eq:SP} is equivalent to
\begin{multline*}
 \frac{\partial \psi_2}{\partial
\varepsilon}(0,q(t))\cdot
\frac{d}{dt}\partial_{3} L\left(t,q(t),\dot{q}(t),{_a^CD_t^\alpha q(t)}\right)\\
+ \partial_{3}L\left(t,q(t),\dot{q}(t),{_a^CD_t^\alpha
q(t)}\right)\cdot
\frac{d}{dt}\frac{\partial \psi_2}{\partial \varepsilon}(0,q(t))\\
+ \partial_{4} L\left(t,q(t),\dot{q}(t),{_a^CD_t^\alpha q(t)}\right)
\cdot {_a^CD_t^\alpha \frac{\partial \psi_2}{\partial \varepsilon}(0,q(t))}\\
- \frac{\partial \psi_2}{\partial \varepsilon}(0,q(t))\cdot
{_tD_b^\alpha}\partial_{4} L\left(t,q(t),\dot{q}(t),{_a^CD_t^\alpha
q(t)}\right) = 0 \, .
\end{multline*}
The proof is completed.
\end{proof}

\begin{theorem}(Transfer formula \cite{CD:BoCrGr}).
\label{trfo} \\Consider functions $f,g\in
C^{\infty}\left([a,b];\mathbb{R}^n\right)$ and assume the following
condition $(\mathcal{C})$: the sequences $\left(g^{(k)}\cdot
_aI_t^{k-\alpha}(f-f(a))\right)_{k\in \mathbb{N}\setminus\{0\}}$ and
$\left(f^{(k)}\cdot _tI_b^{k-\alpha}g\right)_{k\in
\mathbb{N}\setminus\{0\}}$ converge uniformly to $0$ on $[a,b]$.
Then, the following equality holds:
\begin{multline*}
g\cdot {_a^CD_t^\alpha}f-f\cdot{_tD_b^\alpha}g\\
=\frac{d}{dt}\left[\sum_{r=0}^{\infty}\left((-1)^{r}g^{(r)} \cdot
{_aI_t}^{r+1-\alpha}(f-f(a))+f^{(r)} \cdot
{_tI_b}^{r+1-\alpha}g\right)\right].
\end{multline*}
\end{theorem}

\begin{theorem}(Fractional Noether's theorem without transformation of time).
\label{theo:tnadf1} If functional \eqref{Pf} is invariant in the
sense of Definition~\ref{def:inv1:MR} and functions $\frac{\partial
\psi_2}{\partial \varepsilon}(0,q)$ and $\partial_{4} L$ satisfy
condition $(\mathcal{C})$ of Theorem~\ref{trfo}, then
\begin{multline}\label{TN}
\frac{d}{dt}\Biggl[f_2\cdot\partial_{3} L
+\sum_{r=0}^{\infty}\Bigl((-1)^{r}\partial_{4} L^{(r)}
\cdot {_aI_t}^{r+1-\alpha}(f_2-f_2(a))\\
+f_2^{(r)}\cdot {_tI_b}^{r+1-\alpha}\partial_{4} L\Bigr)\Biggr] = 0
\end{multline}
along any fractional $S_h$-extremal with classical and Caputo
derivatives $q(\cdot)$, $t \in [a,b]$ (Definition~\ref{df1}). In
\eqref{TN} $f_2$ denote $\frac{\partial \psi_2}{\partial
\varepsilon}(0,q)\,.$
\end{theorem}

\begin{proof}
We combine equation \eqref{eq:cnsidf11} and Theorem~\ref{trfo}.
\end{proof}

The next definition gives a more general notion
of invariance for the integral functional \eqref{Pf}.
The main result of this section, the
Theorem~\ref{theo:tndf}, is formulated
with the help of this definition.

\begin{definition}(Invariance of \eqref{Pf}).
\label{def:invadf}
 Functional \eqref{Pf} is said to be
$\varepsilon$-invariant under the action of one parameter group of
diffeomorphisms
$\Psi_{i=1,2}=\left\{\psi_i(\varepsilon,\cdot)\right\}_{\varepsilon\in\mathbb{R}}$
of $\mathbb{R}^{1+n}$ if it satisfies for any solution $q(\cdot)$ of
\eqref{eq:eldf}
\begin{multline}
\label{eq:invdf1}
\int_{t_{a}}^{t_{b}} L\left(t,q(t),\dot{q}(t),{_a^CD_t^\alpha q(t)}\right) dt\\
= \int_{\psi_1(\varepsilon,t_{a})}^{\psi_1(\varepsilon,t_{b})}
L\left(\psi_1(\varepsilon,t),\psi_2(\varepsilon,q(t)),\frac{\dot{\psi_2}(\varepsilon,q(t))}{\dot{\psi_1}(\varepsilon,t)},{_{\bar{a}}^CD_{\bar{t}}^\alpha
\psi_2(\varepsilon,q(t))}\right)\dot{\psi_1}(\varepsilon,q(t))dt
\end{multline}
for any subinterval $[{t_{a}},{t_{b}}] \subseteq [a,b]$, where $\dot{\psi_i}=\frac{d\psi_i}{dt}$, $i=1,2$, $\bar{a}=\psi_1(\varepsilon,t_{a})$ and $\bar{t}=\psi_1(\varepsilon,t)\,.$
\end{definition}

Our next result gives a general form of Noether's theorem for
fractional problems of the calculus of variations with classical and
Caputo derivatives.

\begin{theorem}(Fractional Noether's theorem with classical and Caputo derivatives).
\label{theo:tndf} If functional \eqref{Pf} is invariant, in the
sense of Definition~\ref{def:invadf}, and functions $\frac{\partial
\psi_2}{\partial \varepsilon}(0,q)$ and $\partial_{4} L$ satisfy
condition $(\mathcal{C})$ of Theorem~\ref{trfo}, then
\begin{multline}
\label{eq:LC:Frac:RL1} \frac{d}{dt}\Biggl[f_2\cdot\partial_{3} L
+\sum_{r=0}^{\infty}\Bigl((-1)^{r}\partial_{4} L^{(r)}
\cdot {_aI_t}^{r+1-\alpha}(f_2-f_2(a))\\
+f_2^{(r)}\cdot {_tI_b}^{r+1-\alpha}\partial_{4} L\Bigr)\\
+ \tau\Bigl(L-\dot{q}\cdot\partial_{3} L
-\alpha\partial_{4} L\cdot{_a^CD_t}^{\alpha}q\Bigr)\Biggr]= 0
\end{multline}
along any fractional $S_{h}$-extremal with classical and Caputo
derivatives $q(\cdot)$, $t \in [a,b]\,.$ Here and the sequel $f_2$
and $\tau$ denote $\frac{\partial \psi_2}{\partial
\varepsilon}(0,q)$ and $\frac{\partial \psi_1}{\partial
\varepsilon}(0,t)$, respectively.
\end{theorem}

\begin{proof}
Our proof is an extension of the method used in \cite{CD:Jost:1998}.
For that we reparametrize the time (the independent variable $t$) by
the Lipschitz transformation
\begin{equation*}
[a,b]\ni t\longmapsto \sigma f(\lambda) \in [\sigma_{a},\sigma_{b}]
\end{equation*}
that satisfies
\begin{equation}
\label{eq:condla} t_{\sigma}^{'} =\frac{dt(\sigma)}{d\sigma}=
f(\lambda) = 1\,\, if\,\, \lambda=0\,.
\end{equation}
Functional \eqref{Pf} is reduced, in this way, to an autonomous
functional:
\begin{multline}
\label{eq:tempo}
\bar{I}[t(\cdot),q(t(\cdot))]\\
= \int_{\sigma_{a}}^{\sigma_{b}}
L\left(t(\sigma),q(t(\sigma)),\dot{q}(t(\sigma)),
{_{\sigma_{a}}^CD_{t(\sigma)}^{\alpha}q(t(\sigma))}
\right)t_{\sigma}^{'} d\sigma ,
\end{multline}
where $t(\sigma_{a}) = a$ and $t(\sigma_{b}) = b$. Using the
definitions and properties of fractional derivatives given in
Section~\ref{sec:fdRL}, we get successively that
\begin{equation*}
\begin{split}
_{\sigma_{a}}^C&D_{t(\sigma)}^{\alpha}q(t(\sigma))\hspace*{-0.5cm}\\
&= \frac{1}{\Gamma(1-\alpha)} \int_{\frac{a}{f(\lambda)}}^{\sigma
f(\lambda)}\left({\sigma
f(\lambda)}-\theta\right)^{-\alpha}\frac{d}{d\theta}
q\left(\theta f^{-1}(\lambda)\right)d\theta    \\
&= \frac{(t_{\sigma}^{'})^{-\alpha}}{\Gamma(1-\alpha)}
\int_{\frac{a}{(t_{\sigma}^{'})^{2}}}^{\sigma}
(\sigma-s)^{-\alpha}\frac{d}{ds}
q(s)ds  \\
&= (t_{\sigma}^{'})^{-\alpha}\,\,{{^C_\chi
D_{\sigma}^{\alpha}q(\sigma)}},\,\left(
\chi={\frac{a}{(t_{\sigma}^{'})^{2}}}\right)\, .
\end{split}
\end{equation*}

We then have
\begin{equation*}
\begin{split}
&\bar{I}[t(\cdot),q(t(\cdot))] \\
&= \int_{\sigma_{a}}^{\sigma_{b}}
L\left(t(\sigma),q(t(\sigma)),\frac{q_{\sigma}^{'}}{t_{\sigma}^{'}},(t_{\sigma}^{'})^{-\alpha}\,\,{^C_\chi
D_{\sigma}^{\alpha}q(\sigma)}\right) t_{\sigma}^{'} d\sigma \\
&\doteq \int_{\sigma_{a}}^{\sigma_{b}}
\bar{L}_{f}\left(t(\sigma),q(t(\sigma)),q_{\sigma}^{'},t_{\sigma}^{'},{^C_\chi
D_{\sigma}^{\alpha}} q(t(\sigma))\right)d\sigma \\
&= \int_a^b L\left(t,q(t),\dot{q}(t),{_a^CD_t}^{\alpha} q(t)\right) dt \\
&= I[q(\cdot)] \, .
\end{split}
\end{equation*}
If the integral functional \eqref{Pf} is invariant in the sense of
Definition~\ref{def:invadf}, then the integral functional
\eqref{eq:tempo} is invariant in the sense of
Definition~\ref{def:inv1:MR}. It follows from
Theorem~\ref{theo:tnadf1} that
\begin{multline}
\label{eq:tnadf2} \frac{d}{dt}\Biggl[f_2\cdot\partial_{3}
\bar{L}_{f}+\tau\frac{\partial}{\partial t'_\sigma} \bar{L}_{f}
+\sum_{r=0}^{\infty}\Bigl((-1)^{r}
\partial_{5} \bar{L}_{f}^{(r)}\cdot {_aI_t}^{r+1-\alpha}(f_2-f_2(a))\\
+f_2^{(r)}\cdot {_tI_b}^{r+1-\alpha}\partial_{5}
\bar{L}_{f}\Bigr)\Biggr]= 0 \, .
\end{multline}
For $\lambda = 0$, the condition \eqref{eq:condla} allow us to write
that
\begin{equation*}
{^C_\chi D_{\sigma}^{\alpha}q(\sigma)}
 = {_a^CD_t}^{\alpha} q(t)
\end{equation*}
and, therefore, we get
\begin{equation}
\label{eq:prfMR:q1}
\begin{cases}
\partial_{3}\bar{L}_{f}
=\partial_{3} L,\\
\partial_{5}\bar{L}_{f}
=\partial_{4} L,
\end{cases}
\end{equation}
and
\begin{equation}
\label{eq:prfMR:q2}
\begin{split}
&\frac{\partial}{\partial t'_\sigma} \bar{L}_{f} = L +
\partial_{3}{\bar{L}_{f}} \cdot t_{\sigma}^{'}
\frac{\partial}{\partial t_{\sigma}^{'}}
\frac{q_{\sigma}^{'}}{t_{\sigma}^{'}}+\partial_{4}{\bar{L}_{f}}\\
&\times \frac{\partial}{\partial t_{\sigma}^{'}}\left[
\frac{(t_{\sigma}^{'})^{-\alpha}}{\Gamma(1-\alpha)}
\int_{\frac{a}{(t_{\sigma}^{'})^{2}}}^{\sigma}
(\sigma-s)^{-\alpha}\frac{d}{ds} q(s)ds\right]t_{\sigma}^{'}\\
&=-\dot{q}\cdot\partial_{3} L -\alpha\partial_{4}
L\cdot{_a^CD_t}^{\alpha}q + L \, .
\end{split}
\end{equation}
We obtain \eqref{eq:LC:Frac:RL1} substituting \eqref{eq:prfMR:q1}
and \eqref{eq:prfMR:q2} into equation \eqref{eq:tnadf2}.
\end{proof}

Theorem~\ref{theo:tndf} gives a new and interesting result for autonomous fractional
variational problems. Let us consider an autonomous fractional variational problem \textrm{i.e.}, the case when function $L$
 of \eqref{Pf}
do not depends explicitly on the independent variable $t$:
\begin{gather}
I[q(\cdot)] =\int_a^b L\left(q(t),\dot{q}(t),{_a^CD_t}^{\alpha} q(t)\right) dt
\longrightarrow \min \, . \label{eq:FOCP:CO4}
\end{gather}

\begin{corollary}
\label{cor:FOCP:CV}
For the autonomous fractional problem \eqref{eq:FOCP:CO4} one has
\begin{equation*}
 \frac{d}{dt}\Bigl(L-\dot{q}\cdot\partial_{3} L
-\alpha\partial_{4} L\cdot{_a^CD_t}^{\alpha}q\Bigr)= 0
\end{equation*}
along any fractional $S_{h}$-extremal with classical and Caputo
derivatives $q(\cdot)$, $t \in [a,b]\,.$
\end{corollary}

\begin{proof}
As the Lagrangian  $L$ does not depend explicitly on the independent
variable $t$, we can easily see that \eqref{eq:FOCP:CO4} is
invariant under translation of the time variable: the condition of
invariance \eqref{eq:invdf1} is satisfied with
$\psi_1(\varepsilon,t) = t+\varepsilon$ and
$\psi_2(\varepsilon,q(t))= q(t)$. Indeed, given that
$\frac{d\psi_1}{dt}(\varepsilon,q(t)) = 1$, $\tau=1$ and $f_2=0$, the invariance condition
\eqref{eq:invdf1} is verified if ${_{\bar{a}}^CD_{\bar{t}}^\alpha
\psi_2(\varepsilon,q(t))} = {_a^CD_t^\alpha} q(t)$. This is true
because
\begin{equation*}
\begin{split}
{_{\bar{a}}^CD_{\bar{t}}^\alpha
\psi_2(\varepsilon,q(t))}
&=\frac{1}{\Gamma(1-\alpha)} \int_{\bar{a}}^{\bar{t}}
(\bar{t}-\theta)^{-\alpha}\frac{d}{d\theta}
\psi_2(\varepsilon,q(\theta))d\theta
\\
&=\frac{1}{\Gamma(1-\alpha)} \int_{a+\varepsilon}^{t+\varepsilon}
(t+\varepsilon-\theta)^{-\alpha}\frac{d}{d\theta}
\psi_2(\varepsilon,q(\theta))d\theta\\
&=\frac{1}{\Gamma(1-\alpha)} \int_{a}^{t}
(t-s)^{-\alpha}\frac{d}{ds}
\psi_2(\varepsilon,q(s+\varepsilon))ds\\
&={_a^CD^\alpha_{t}\psi_2(\varepsilon,q(t+\varepsilon))}={_a^CD^\alpha_{t}\psi_2(\varepsilon,q(t}))\\
&={_a^CD^\alpha_{t}{q}(t)}\, .
\end{split}
\end{equation*}
\end{proof}

\begin{remark}
If $\alpha=1$ Problem \eqref{eq:FOCP:CO4} is reduced to the
classical problem of the calculus of variations,
\begin{equation*}
\label{eq:pbcv11}
 I[q(\cdot)] = \int_a^b
F\left(q(t),\dot{q}(t)\right) \longrightarrow \min
\end{equation*}
with $F\left(q(t),\dot{q}(t)\right):=L\left(q(t),\dot{q}(t),\dot{q}(t)\right)$, and one obtains from Corollary~\ref{cor:FOCP:CV} the famous conservation of energy in classical mechanics:
$$F-\dot{q}\cdot \frac{\partial F}{\partial \dot{q}}=constant$$ along any solutions of the equations \eqref{eq:EL}.
\end{remark}
% ------------------------------------------------

\section{Noether's theorem for the linear friction problem}
\label{example}

In order to formulate an action principle for dissipative systems free from the problems found in the original Riewe's approach, in a recent work \cite{LazoCesar} it was proposed that the equation of motion for dissipative systems can be obtained by taking the limit $a\rightarrow b$ with $t=a+(b-a)/2=(a+b)/2$ in the extremal of the action
\begin{equation}
\label{b8}
I[q(\cdot)]=\int_a^b {L}\left(q(t),\dot{q}(t),{_a^C D^{\alpha}_t} q(t)\right) dt.
\end{equation}
Furthermore, it was proposed a quadratic Lagrangian for a particle under a frictional force proportional to the velocity as \cite{LazoCesar}
\begin{equation}
\label{b10}
L\left(q(t),\dot{q}(t),{_a^C D^{\frac{1}{2}}_t} q(t)\right)=\frac{1}{2}m\left(\dot{q}(t)\right)^2-U(q(t))+\frac{\gamma}{2}\left({_a^C D^{\frac{1}{2}}_t} q(t)\right)^2,
\end{equation}
where the three terms in \eqref{b10} represent the kinetic energy, potential energy, and the fractional linear friction energy, respectively. Since the equation of motion is obtained in the limit $a\rightarrow b$, if we consider the last term in \eqref{b10} up to first order in $\Delta t=b-a$ we get:
\begin{equation}
\label{b11}
\frac{\gamma}{2}\left({_a^C D^{\frac{1}{2}}_t} q\right)^2\approx \frac{\gamma}{2}\left(\frac{\Gamma(1)}{\Gamma(\frac{3}{2})}\right)^2\left(\dot{q}\right)^2\Delta t \approx \frac{2}{\pi} \gamma \dot{q} \Delta q,
\end{equation}
that coincide, apart from the multiplicative constant $2/\pi$, with the work from the frictional force $\gamma \dot{q}$ in the displacement $\Delta q \approx \dot{q}\Delta t$. This additional constant is a consequence of the use of fractional derivatives in the Lagrangian and do not appears in the equation of motion after we apply the action principle \cite{LazoCesar}. Furthermore, the Lagrangian \eqref{b10} is physical in the sense it provide us with physically meaningful relations for the momentum and the canonical Hamiltonian \cite{LazoCesar}
\begin{equation}
\label{b14}
H=p\dot{q}+p_{\frac{1}{2}}{_a^C D^{\frac{1}{2}}_t} q-L=\frac{1}{2}m\left(\dot{q}\right)^2+U(q)+\frac{\gamma}{2}\left({_a^C D^{\frac{1}{2}}_t} q\right)^2,
\end{equation}
where
\begin{equation}
\label{b13}
p=\frac{\partial L}{\partial \dot{q}}=m\dot{q}, \;\;\; p_{\frac{1}{2}}=\frac{\partial L}{\partial {_a^C D^{\frac{1}{2}}_t} q}=\gamma{_a^C D^{\frac{1}{2}}_t} q.
\end{equation}
From \eqref{b13} and \eqref{b14} we can see that the Lagrangian \eqref{b10} is physical in the sense it provides us a correct relation for the momentum $p_1=m\dot{q}$, and a physically meaningful Hamiltonian (it is the sum of all energies). Furthermore, the additional fractional momentum $p_{\frac{1}{2}}=\gamma{_a^C D^{\frac{1}{2}}_t} q$ goes to zero when we takes the limit $a\rightarrow b$ \cite{LazoCesar}.

Finally, the equation of motion for the particle is obtained by inserting our Lagrangian \eqref{b10} into the Euler-Lagrange equation \eqref{eq:eldf},
\begin{equation}
\label{b15}
m \ddot{q}-\gamma{_t D^{\frac{1}{2}}_b} {_a^C D^{\frac{1}{2}}_t}q=F(q),
\end{equation}
where $F(q)=-\frac{d}{dq}U(q)$ is the external force. By taking the limit $a\rightarrow b$ with $t=(a+b)/2$ and using the approximation ${_a^C D^{\frac{1}{2}}_t}q \approx -{_t^C D^{\frac{1}{2}}_b}q$ \cite{LazoCesar} we obtain the equation of motion for a particle under a linear friction force
\begin{equation}
\label{b16}
m \ddot{q}+\gamma \dot{q}=F(q).
\end{equation}

Finally, Noether's invariant theorems states that if an action remains invariant with respect to a group of transformations, such transformations leads to a corresponding conservation law. Since for the Lagrangian \eqref{b10} the linear friction is an autonomous fractional problem, Corollary \ref{cor:FOCP:CV} gives us
\begin{equation}
\label{b17}
 \frac{d}{dt}\Bigl(L-p\dot{q}
-\frac{1}{2}p_{\frac{1}{2}}\cdot{_a^CD_t}^{\frac{1}{2}}q\Bigr)= \frac{d}{dt}\Bigl(\frac{1}{2}p_{\frac{1}{2}}\cdot{_a^CD_t}^{\frac{1}{2}}q-H\Bigr) = \frac{d}{dt}\Bigl(\frac{\gamma}{2}\left({_a^CD_t}^{\frac{1}{2}}q\right)^2-H\Bigr) = 0.
\end{equation}
From \eqref{b17} it is ease to see that the Hamiltonian for a particle under frictional forces is not a conserved quantity, as expected. The Hamiltonian and consequently the total energy of the system is only time locally conserved, when we consider only very short time intervals by taking the limit $a\rightarrow b$. In this last case we have ${_a^CD_t}^{\frac{1}{2}}q\rightarrow 0$ and \eqref{b17} reduces to $ \frac{dH}{dt} = 0$.
% ------------------------------------------------

\section{Fractional optimal control problems with classical and Caputo derivatives}
\label{foc}

We now adopt the Hamiltonian formalism in order to generalize the
Noether type results found in
\cite{CD:Djukic:1972,Caputo,CD:JMS:Torres:2002a} for the more general context
of fractional optimal control problems with classical and Caputo derivatives.
For this, we make use of our Noether's Theorem~\ref{theo:tndf} and
the standard Lagrange multiplier technique (\textrm{cf.}
\cite{CD:Djukic:1972}). The fractional optimal control problem with classical
and Caputo derivatives is introduced, without loss of
generality, in Lagrange form as in \cite{AlmeidaTorres,PAT}:
\begin{equation}
\label{eq:COA}
I[q(\cdot),u(\cdot),\mu(\cdot)]
= \int_a^b L\left(t,q(t),u(t),\mu(t)\right) dt \longrightarrow \min
\end{equation}
subject to the differential system
\begin{gather}
\label{eq:sitRL1}
\dot{q}(t)=\varphi\left(t,q(t),u(t)\right),\\
\label{eq:sitRL}
 {_a^CD_t}^\alpha
q(t)=\rho\left(t,q(t),\mu(t)\right)
\end{gather}
and initial condition
\begin{equation}
\label{eq:COIRL}
q(a)=q_a\, .
\end{equation}

The Lagrangian $L :[a,b] \times \mathbb{R}^{n}\times
\mathbb{R}^{m}\times \mathbb{R}^{d}\rightarrow \mathbb{R}$, the
velocity vector $\varphi:[a,b] \times \mathbb{R}^{n}\times
\mathbb{R}^m\rightarrow \mathbb{R}^{n}$ and the fractional velocity
vector $\rho:[a,b] \times \mathbb{R}^{n}\times
\mathbb{R}^d\rightarrow \mathbb{R}^{n}$ are assumed to be functions
of class $C^{1}$ with respect to all their arguments. We also
assume, without loss of generality, that $0<\alpha<1$. In conformity
with the calculus of variations, we are considering that the control
functions $u(\cdot)$ and $\mu(\cdot)$ take values on an open set of
$\mathbb{R}^m$ and $\mathbb{R}^d$, respectively.

\begin{remark}
\label{rem:cv:pc}
The fractional functional of the calculus of variations
with classical and Caputo derivatives \eqref{Pf}
is obtained from \eqref{eq:COA}--\eqref{eq:sitRL} choosing\\
$\varphi(t,q,u)=u$ and $\rho(t,q,\mu)=\mu$.
\end{remark}

% ------------------------------------------------

\subsubsection{Fractional Pontryagin Maximum Principle}

In the fifties of the twentieth century, L.S. Pontryagin and his collaborators proved the main necessary optimality condition for optimal control
problems: the famous Pontryagin Maximum Principle \cite{Pontryagin}.

 In this subsection we prove a fractional maximum principle
with the help of optimality conditions \eqref{eq:eldf}.

\begin{definition}(Process with classical and Caputo derivatives).
\label{pros}
An admissible triplet $(q(\cdot),u(\cdot),\mu(\cdot))$ that satisfies the
control system \eqref{eq:sitRL1}--\eqref{eq:sitRL} of the
optimal control problem \eqref{eq:COA}--\eqref{eq:COIRL},
$t \in [a,b]$, is said to be a \emph{process with classical and Caputo derivatives}.
\end{definition}

For convenience of notation, we introduce the following operator:
$$
[q,u,\mu,p,p_{\alpha}](t) = \left(t, q(t), u(t), \mu(t), p(t), p_{\alpha}(t)\right)
$$

\begin{remark}
In mechanics, $p(\cdot)$ and $p_\alpha(\cdot)$ correspond to the generalized
momentum related to $\dot{q}(\cdot)$ and ${_a^CD_t}^\alpha
q(\cdot)$, respectively. In the language of optimal control $p(\cdot)$ and
$p_\alpha(\cdot)$ are called the adjoint variables.
\end{remark}

\begin{theorem}(Fractional Pontryagin Maximum Principle).
\label{th:AG}
If\\ $(q(\cdot),u(\cdot),\mu(\cdot))$ is a process
for problem \eqref{eq:COA}--\eqref{eq:COIRL}, in the sense of Definition~\ref{pros},
then there exists co-vector functions $p(\cdot)\in PC^{1}([a,b];\mathbb{R}^{n})$
and $p_{\alpha}(\cdot)\in PC^{1}([a,b];\mathbb{R}^{n})$
such that for all $t\in [a,b]$ the
quadruple  $(q(\cdot),u(\cdot),p(\cdot),p_{\alpha}(\cdot))$ satisfies
the following conditions:
\begin{itemize}
\item the Hamiltonian system
\begin{equation}
\label{eq:HamRL}
\begin{cases}
\partial_5 {\cal H}[q,u,\mu,p,p_{\alpha}](t)=\dot{q}(t)\,  ,\\
\partial_6 {\cal H}[q,u,\mu,p,p_{\alpha}](t)={_a^CD_t}^\alpha q(t)  \, , \\
\partial_2{\cal H}[q,u,\mu,p,p_{\alpha}](t)=-\dot{p}(t)+_tD_b^\alpha p_{\alpha}(t)  \, ;
\end{cases}
\end{equation}
\item the stationary conditions
\begin{equation}
\label{eq:HamRL1}
\begin{cases}
 \partial_3 {\cal H}[q,u,\mu,p,p_{\alpha}](t)=0 \, ,\\
 \partial_4 {\cal H}[q,u,\mu,p,p_{\alpha}](t)=0 \, ;
\end{cases}
\end{equation}
\end{itemize}
where the Hamiltonian ${\cal H}$ is given by
\begin{multline}
\label{eq:HL}
{\cal H}\left[t,q,u,\mu,p,p_{\alpha}\right](t) \\
= L\left(t,q(t),u(t),\mu(t)\right) + p(t) \cdot \varphi\left(t,q(t),u(t)\right)
+p_{\alpha}(t)\cdot \rho(t,q(t),\mu(t)) \, .
\end{multline}
\end{theorem}

\begin{proof}
Minimizing \eqref{eq:COA} subject to
\eqref{eq:sitRL1}--\eqref{eq:sitRL} is equivalent,
by the Lagrange multiplier rule,
to minimize the augmented functional $J[q(\cdot),u(\cdot),\mu(\cdot),p(\cdot),p_{\alpha}(\cdot)]$
defined by
\begin{multline}
\label{eq:COA1} J[q(\cdot),u(\cdot),\mu(\cdot),p(\cdot),p_{\alpha}(\cdot)]
= \int_a^b \Bigl[{\cal H}[q,u,\mu,p,p_{\alpha}](t)\\
- p(t) \cdot \dot{q}(t)-p_{\alpha}(t)\cdot {_a^CD_t}^\alpha q(t)\Bigr]dt
\end{multline}
with ${\cal H}$ given by \eqref{eq:HL}.

Theorem~\ref{th:AG} is proved applying the necessary optimality
condition \eqref{eq:eldf} to the augmented functional
\eqref{eq:COA1}: we only proof the one of optimality equations of
Theorem~\ref{th:AG} (the reasoning is similar for the other
equations)
\begin{multline*}
\partial_{2} \mathcal{L}[q,u,\mu,p,p_{\alpha}](t)
-\frac{d}{dt}\frac{\partial\mathcal{L}}{\partial \dot{q}}
[q,u,\mu,p,p_{\alpha}](t) \\+
{_tD_b^\alpha}\frac{\partial\mathcal{L}}{\partial\,{_a^CD_t^\alpha}
q}
[q,u,\mu,p,p_{\alpha}](t)  = 0\\
\Leftrightarrow \partial_2{\cal
H}[q,u,\mu,p,p_{\alpha}](t)=-\dot{p}+_tD_b^\alpha p_{\alpha}
\end{multline*}
where $$\mathcal{L}[q,u,\mu,p,p_{\alpha}](t)={\cal H}[q,u,\mu,p,p_{\alpha}](t)\\
- p(t) \cdot \dot{q}(t)-p_{\alpha}(t)\cdot {_a^CD_t}^\alpha
q(t)\,.$$
\end{proof}

\begin{definition}(Pontryagin $S_h$-extremal with classical
and fractional derivatives). \label{PR} A tuple
$\left(q(\cdot),u(\cdot),\mu(\cdot),p(\cdot),p_{\alpha}(\cdot)\right)$
satisfying Theorem~\ref{th:AG} is called a \emph{Pontryagin
$S_h$-extremal with classical and Caputo derivatives}.
\end{definition}

\begin{remark}
For problems of the calculus of variations with classical and Caputo
derivatives, one has $\varphi(t,q,u)=u$ and $\rho(t,q,\mu)=\mu$
(Remark~\ref{rem:cv:pc}). Therefore, ${\cal H} = L + p \cdot u
+p_{\alpha}\cdot\mu$. From the Hamiltonian system \eqref{eq:HamRL}
we get
\begin{equation}
\begin{cases}\label{edr}
 u=\dot{q}\\
\mu={_a^CD_t}^\alpha q\\
\partial_2L=-\dot{p}+_tD_b^\alpha p_{\alpha}
\end{cases}
\end{equation}
and from the stationary conditions \eqref{eq:HamRL1}
\begin{equation}
\begin{cases}\label{edr1}
\partial_3{\cal H}=0\Leftrightarrow \partial_3 L=-p
\Rightarrow\frac{d}{dt}\partial_3 L=-\dot{p},\\
\partial_4{\cal H}=0\Leftrightarrow \partial_4 L=-p_{\alpha}
\Rightarrow {_tD_b^\alpha}\partial_4 L=-{_tD_b^\alpha} p_{\alpha}.
\end{cases}
\end{equation}
Substituting the quantities \eqref{edr1} into \eqref{edr}, we arrive
to the Euler--Lagrange equations with classical and Caputo
derivatives \eqref{eq:eldf}.
\end{remark}

% ------------------------------------------------

\subsubsection{Noether's theorem for fractional optimal control problems}

The notion of variational invariance for
\eqref{eq:COA}--\eqref{eq:sitRL} is defined with the help of the
augmented functional \eqref{eq:COA1}.

\begin{definition}(Variational invariance of \eqref{eq:COA}--\eqref{eq:sitRL}).
\label{def:inv:gt1} The augmented functional \eqref{eq:COA1} is said
to be $\varepsilon$-invariant under the action of one parameter
group of diffeomorphisms
$$\Psi_{i=1,...,6}=\left\{\psi_i(\varepsilon,\cdot)\right\}_{\varepsilon\in\mathbb{R}}\in
\mathbb{R}\times
\mathbb{R}^n\times\mathbb{R}^m\times\mathbb{R}^d\times\mathbb{R}^n\times\mathbb{R}^n$$
 if it satisfies for any \emph{Pontryagin $S_h$-extremal
with classical and Caputo derivatives}
\begin{multline}
\label{eq:condInv} \bigl({\cal
H}\left(\psi_1(\varepsilon,t),\psi_2(\varepsilon,q(t)),
\psi_3(\varepsilon,u(t)),\psi_4(\varepsilon,\mu(t)),\psi_5(\varepsilon,p(t)),\psi_6(\varepsilon,p_{\alpha}(t))\right) \\
-\psi_5(\varepsilon,p(t))\cdot \frac{\dot{\psi_2}(\varepsilon,q(t))}{\dot{\psi_1}(\varepsilon,t)}
-\psi_6(\varepsilon,p_{\alpha}) \cdot
{_{\bar{a}}^CD_{\bar{t}}^\alpha
\psi_2(\varepsilon,q(t))})\bigr)\dot{\psi_1}(\varepsilon,t) \\
=\left({\cal H}[q,u,\mu,p,p_{\alpha}](t) -p(t)\cdot
\dot{q}(t)-p_{\alpha}(t)\cdot{_a^CD_t}^\alpha q(t)\right)\,.
\end{multline}
for any subinterval $[{t_{a}},{t_{b}}] \subseteq [a,b]$.
\end{definition}

In \cite{CD:BoCrGr}, the author proved the Noether's theorem without
transformation of the independent variable $t$ for the following
fractional control problem: .
\begin{gather*}
 I[q(\cdot),u(\cdot)] = \int_a^b
L\left(t,q(t),u(t)\right) dt
\longrightarrow \min \\
_a^CD_t^\alpha q(t)=\varphi\left(t,q(t),u(t)\right) \, . \notag
\end{gather*}
In this case he only obtain  the conservation of momentum.

 Next theorem provides an extension of Noether's theorem in general form to the
wider fractional context of optimal control problems with classical
and Caputo derivatives.

\begin{theorem}(Noether's theorem in Hamiltonian form
for optimal control problems with classical and Caputo derivatives).
\label{thm:mainResult} If \eqref{eq:COA}--\eqref{eq:sitRL} is
variationally invariant, in the sense of
Definition~\ref{def:inv:gt1}, and functions $f_2$ and $p_{\alpha}$
satisfy condition $(\mathcal{C})$ of Theorem~\ref{trfo}, then
\begin{multline}
\label{eq:tndf:CO} \frac{d}{dt}\Biggl[-f_2\cdot p
-\sum_{r=0}^{\infty}\Bigl((-1)^{r}p_{\alpha}^{(r)}
\cdot {_aI_t}^{r+1-\alpha}(f_2-f_2(a))\\
+f_2^{(r)} \cdot {_tI_b}^{r+1-\alpha}p_{\alpha}\Bigr) +
\tau\left({\cal H}-(1-\alpha)p_{\alpha}
\cdot{_a^CD_t^{\alpha}}q\right)\Biggr]= 0
\end{multline}
along any Pontryagin $S_h$-extremal with classical and Caputo
derivatives (Definition~\ref{PR}).
\end{theorem}

\begin{proof}
The fractional conservation law \eqref{eq:tndf:CO} is obtained by
applying Theorem~\ref{theo:tndf} to the equivalent functional
\eqref{eq:COA1}.
\end{proof}

Like  Teorem~\ref{theo:tndf}, Theorem~\ref{thm:mainResult} also
gives an interesting result for autonomous fractional problems. Let
us consider an autonomous fractional optimal control problem,
\textrm{i.e.}, the case when functions $L$, $\varphi$ and $\rho$ of
\eqref{eq:COA}--\eqref{eq:sitRL} do not depend explicitly on the
independent variable $t$:
\begin{gather}
I[q(\cdot),u(\cdot),\mu(\cdot)] =\int_a^b L\left(q(t),u(t),\mu(t)\right) dt
\longrightarrow \min \, , \label{eq:FOCP:CO}\\
\dot{q}(t)=\varphi(q(t),u(t))\label{eq:FOCP:CO10}\,,\\
{_a^CD_t^{\alpha}} q(t)=\rho\left(q(t),\mu(t)\right)\,.
\label{eq:45}
\end{gather}

\begin{corollary}
\label{cor:FOCP:CO}
For the autonomous fractional problem
\eqref{eq:FOCP:CO}--\eqref{eq:45} one has
\begin{equation}
\label{eq:frac:eng} \frac{d}{dt}\Biggl[ {\cal H}(q(t),
u(t),\mu(t),p(t),p_{\alpha}(t))
-(1-\alpha)p_{\alpha}(t)\cdot{_a^CD_t^{\alpha}}q(t) \Biggr]= 0
\end{equation}
along any Pontryagin $S_h$-extremal with classical and Caputo
derivatives\\
$(q(\cdot),u(\cdot),\mu(\cdot),p(\cdot),p_{\alpha}(\cdot))$.
\end{corollary}

\begin{proof}
The proof is similar of Corollary~\ref{cor:FOCP:CV} to taking into
account the Definition~\ref{def:inv:gt1} applied to the
Problem \eqref{eq:FOCP:CO}--\eqref{eq:45}.
\end{proof}

The Corollary~\ref{cor:FOCP:CO} shows that in contrast with
the classical autonomous problem of optimal control,
for \eqref{eq:FOCP:CO}--\eqref{eq:45}
the Hamiltonian ${\cal H}$ does not define a conservation law.
Instead of the classical equality
$\frac{d}{dt}\left({\cal H}\right)=0$, we have
\begin{equation}
\label{eq:ConsHam:alpha} \frac{d}{dt} \left[ {\cal H} +
\left(\alpha-1\right) p_{\alpha} \cdot {_a^CD_t^{\alpha}} q \right]
= 0 \, ,
\end{equation}
\textrm{i.e.}, conservation of the Hamiltonian ${\cal H}$
plus a quantity that depends on the fractional order $\alpha$ of differentiation.
This seems to be explained by violation of the homogeneity
of space-time caused by the fractional
derivatives, when $\alpha\neq 1$. If $\alpha=1$, then
we obtain from \eqref{eq:ConsHam:alpha} the classical result:
the Hamiltonian ${\cal H}$ is preserved along all the
Pontryagin extremals.

\section{Conclusion}
\label{conclusion}

In the present work, we obtained a generalization of the Noether's theorem for Lagrangians depending on mixed classical and Caputo derivatives that can be used to obtain constants of motion for dissipative systems. The Noether's theorem of calculus of variation is one of the most important theorems for physics in the 20th century. It is well know that all conservations
laws in mechanics, \textrm{e.g.}, conservation of energy or conservation of momentum, are directly related to the invariance of the action under a family of transformations. However, the classical Noether's theorem can not yields informations about constants of motion for non-conservative systems since it is not possible to formulate physically meaningful Lagrangians for this kind of systems in classical calculus of variation. On the other hand, in recent years the fractional calculus of variation within Lagrangians depending on fractional derivatives has emerged as an elegant alternative to study non-conservative systems. In this context, the generalization of the Noether's theorem for the fractional calculus of variation is fundamental to investigate the action symmetries for non-conservative systems. As an example of application to non-conservative systems, we study the problem of a particle under a frictional force. In addition, we also obtained Noether's conditions for the fractional optimal control problem.

\begin{acknowledgements}
This work was partially supported by CNPq and CAPES (Brazilian research funding agencies).
\end{acknowledgements}



\end{document}